\documentclass[9pt]{amsart}
\usepackage[alphabetic]{amsrefs}
\usepackage{amsmath,amsthm,amsfonts,amscd,amssymb,epsfig,color}
\usepackage[all]{xy}

\usepackage{amssymb}
\usepackage{epsfig}

\newtheorem{theorem}{Theorem}[section]
\newtheorem{corollary}[theorem]{Corollary}
\newtheorem{lemma}[theorem]{Lemma}
\newtheorem{proposition}[theorem]{Proposition}

\newtheorem{definition/theorem}{Definition/Theorem}
\newtheorem{definition/lemma}{Definition/lemma}

\theoremstyle{definition} \newtheorem{definition}[theorem]{Definition}
\newtheorem{remark}[theorem]{Remark}

\newcommand{\ov}[1]{\overline{#1}}
\newcommand{\op}[1]{\operatorname{#1}}
\newcommand{\ovop}[1]{\ov{\op{#1}}}

\begin{document}

\title{On extensions of the Torelli map }

\author{Angela Gibney}

\begin{abstract}
The divisors on $\ovop{M}_g$ that arise as the pullbacks of ample divisors along any extension of the Torelli map to any toroidal compactification of $\op{A}_g$  form a 2-dimensional extremal face 
of the nef cone of $\ovop{M}_g$, which is explicitly described.
\end{abstract}

\maketitle
\section{introduction} 
   
The moduli spaces $\op{M}_g$, of smooth curves of genus g, and $\op{A}_g$, of $g$-dimensional principally polarized Abelian varieties are related by  a natural embedding $\mathfrak{t_g}: \op{M}_g \rightarrow \op{A}_g,$
that sends the isomorphism class of a smooth genus $g$ curve $C$ to a pair consisting of
the isomorphism class of its Jacobian together with its theta divisor
$\Theta_{g-1}$.  The map is named for Torelli who proved that it is injective \cite{Weil}.

Both $\op{M}_g$ and  $\op{A}_g$ are quasi-projective varieties.   The Deligne-Mumford compactification $\ovop{M}_g$ is obtained by adding nodal curves having a finite number of automorphisms.  While this is often taken to be the best compactification of $\op{M}_g$, 
the closure in any compactification of $\op{A}_g$ of the image of
$\op{M}_g$ under the Torelli map gives an alternative one.    In particular, if the Torelli map extends to a morphism from $\ovop{M}_g$
to that compactification of $\op{A}_g$, then one may be able to say something interesting about the moduli space of curves.  For example,
the Satake compactification $\op{A}^{Sat}_g$ is
the normal projective variety constructed by taking Proj of the graded ring of
Siegel modular forms  \cite{SatakeCompactification},
\cite{FaltingsChai}.  The Torelli map extends to a morphism $f_{\lambda}: \ovop{M}_g \longrightarrow \ovop{A}^{Sat}_g$, and 
 by studying the image $\ovop{M}_g^{Sat}$ in $\op{A}^{Sat}_g$, one can prove  the very interesting and nontrivial fact that 
$\op{M}_g$ contains a complete curve for $g \geq 3$ \cite{Oort}.

There are infinitely many toroidal compactifications $\op{A}_g^{\tau}$ of $\op{A}_g$, each dependent on a
choice of fan $\tau$, which describes a decomposition of the
cone of real positive quadratic forms in $g$ variables (\cite{AMRT}, \cite{FaltingsChai}, \cite{Namikawa}, \cite{Olsson}).     
 Each compactification $\op{A}_g^{\tau}$ contains a common sublocus, $\op{A}_g^{part}$, the partial compactification of $\op{A}_g$ defined by  Mumford.   Also, for each $\tau$, there is a morphism $\eta_{\tau}$ from  $\op{A}_g^{\tau}$ to Satake's compactification 
$A^{Sat}_g$.   The first result of this short note is that any regular extension of the Torelli map from $\ovop{M}_g$ to a compactification of $\op{A}_g$ having these properties is given by a sublinear system of a divisor lying on what will be shown to be a 2-dimensional extremal face of the nef cone of $\ovop{M}_g$:

\begin{theorem}\label{general} Let $X$ be any compactification of $\op{A}_g$ that contains Mumford's partial compactification $\op{A}_g^{part}$ and maps to Satake's compactification $\op{A}_g^{Sat}$.  
Then if $f: \ovop{M}_g \longrightarrow X$ is any extension of the Torelli map and $A$ any ample divisor on $X$, then $f^*(A)$ lies in the interior of the face  $\mathbb{F}= \{\alpha \lambda + \beta (12\lambda -\delta_0) |  \ \alpha, \beta \ge 0\}$.   
\end{theorem}

  Mumford and Namikawa \cite{NamikawaNewCompact}, have shown that  $\mathfrak{t_g}$ extends to a morphism 
$f_{vor}: \ovop{M}_g \longrightarrow \op{A}^{vor}_g$ to the Voronoi compactification of $\op{A}_g$, and very recently Alexeev and Brunyate \cite{AB} have shown that 
the Torelli map 
extends to a morphism 
$f_{per}: \ovop{M}_g \longrightarrow \ovop{A}^{per}_g$,
to the perfect cone compactification.    
By using the morphism $f_{per}$, it is possible to see that every divisor in $\mathbb F$ gives rise to such an extension of the Torelli map:

\begin{theorem}\label{perfect}Let $f_{per}: \ovop{M}_g \longrightarrow \op{A}_g^{per}$ be the Toroidal extension of the Torelli map defined in \cite{AB}.
Then $$\mathbb{F}=f_{per}^*(\operatorname{Nef}(\op{A}_g^{per})).$$
\end{theorem}

In particular,  Theorem \ref{perfect} shows that all divisors in the interior of $\mathbb{F}$ are semi-ample.     As is proved in Corollary \ref{semiample}, it follows from \cite{SBFirst} that if we take $\ovop{M}_g$ to be defined over $\mathbb{C}$, then the nef divisor $12\lambda-\delta_0$ is actually semi-ample for $g \le 11$. To exhibit the nontriviality of these statements, in Remark \ref{bpf} we explain that one cannot show that the divisors on $\mathbb F$ are semi-ample by simply applying the base point free theorem.

 $\op{A}^{per}_g$, $\op{A}^{vor}_g$, and $\op{A}^{cen}_g$ have been the most extensively studied toroidal compactifications.
Here  $\tau=cen$ denotes the central cone decomposition which was constructed as a blowup of $\op{A}_g^{Sat}$ by Igusa in \cite{Igusa}.    In \cite{AlexeevVor}, Alexeev has shown that $\op{A}^{vor}_g$ is
the normalization of the main irreducible component of the moduli space $\ovop{AP}_g$ parametrizing isomorphism classes $(X, \theta)$ of stable semi-abelic pairs.   Shepherd-Barron has shown that for $g \ge 12$, one has that $\op{A}^{per}_g$ is the canonical model for the Satake compactification $\op{A}^{Sat}_g$ \cite{SBFirst}, and of any smooth compactification of $\op{A}_g$ (even when regarded as a stack over $\op{A}_g^{Sat}$).  Shepherd-Barron has remarked in \cite{SBFirst} that he has no reason to believe that $\op{A}^{per}_g$ is a moduli space;  and no other Toroidal compactification is known to be such.    

General relationships between the toroidal compactifications given by different fans are not fully understood.  There is a rational map 
 $g: \op{A}_g^{vor} \dashrightarrow \op{A}_g^{per}$ which is an isomorphism for $g =2$, $3$, a morphism for $g \le 5$, but not regular for $g \ge 6$ (\cite{AB}, \cite{ER1}, \cite{ER2}).  For $g=2$ and $g=3$ all three compactifications coincide, for $g=4$, one has that 
 $\op{A}_4^{per}=\op{A}_4^{cen}$, but are not equal to $\op{A}_4^{vor}$.  For $g \le 7$, that
there is toroidal extension of the Torelli map $f_{cen}:\ovop{M}_g \longrightarrow \op{A}^{cen}_g$, but that there is no such extension for $g \ge 9$ (cf. [VRG10],\cite{AB}).  Although the Picard group of $\op{A}^{vor}_4$  has been found  \cite{HulekSankaranNef} (it has 3 generators), even the number of boundary components of $\op{A}_g^{vor}$ is unknown for general $g$.  In contrast, the Picard group of $\op{A}^{per}_g$
has rank $2$, and the nef cone of $\op{A}^{per}_g$ has been explicitly described \cite{SBFirst}.   

 The result from Theorem \ref{general}, that all pullbacks of ample divisors from toroidal compactifications lie in the face $\mathbb F$, begs the question of whether, in fact, the images of $\ovop{M}_g$ under these extensions are all isomorphic. There is some evidence for this possibility already. 
Alexeev and Brunyate \cite{AB} show that the image of $\ovop{M}_g$ under $f_{vor}$ is isomorphic to its image under their extension of Torelli map 
to the perfect compactification.   One consequence of Theorem~\ref{general} gives at least a rough comparison of the images of $\ovop{M}_g$ under morphisms to them: 

\begin{corollary}\label{Pic}The image of $\ovop{M}_g$ under any extension of the Torelli map to a compactification of $\op{A}_g$ that contains Mumford's partial compactification $\op{A}_g^{part}$ and maps to Satake's compactification $\op{A}_g^{Sat}$ has Picard number $2$.
\end{corollary}

The face $\mathbb{F}$ is really the first region of the nef cone of $\ovop{M}_g$  to have been studied, and in fact, for $g=2$, one has that  $\mathbb{F} = \op{Nef}(\ovop{M}_2)$.   In \cite[Prop. 3.3]{FaberThesis}, Faber used \cite[Thm. 1.3]{CornalbaHarris} to establish a base case of an induction to show $12\lambda-\delta_0$ is nef for $g \ge 2$.    The divisor $\lambda$ is semi-ample -- and responsible for the morphism from $\ovop{M}_g$ to $\op{A}_g^{Sat}$.     For $g=3$, 
a cross-section of the corresponding chamber $Ch(\mathbb{F})$ of the effective cone adjacent to $\mathbb{F}$ is depicted on the left in the figure above.   The image on the right shows the chambers of the effective cone that are known for  $\ovop{M}_3$.     
For general $g \ge 2$, the chamber $Ch(\lambda)$, directly adjacent to $Ch(\mathbb{F})$, corresponds to the morphism $f_{Sat}$.   For $g \ge 3$ the chambers on the other side of $Ch(\mathbb{F})$ have been unearthed by the Hassett-Keel program (\cite{Hassettg2}, \cite{HHFlip}, \cite{HassettHyeonDivCon}, \cite{HyeonLeeHyperelliptic}, \cite{HyeonLee3}, \cite{HyeonLee2}).

  \begin{figure}
{\includegraphics[width=3cm,height=2.5cm]{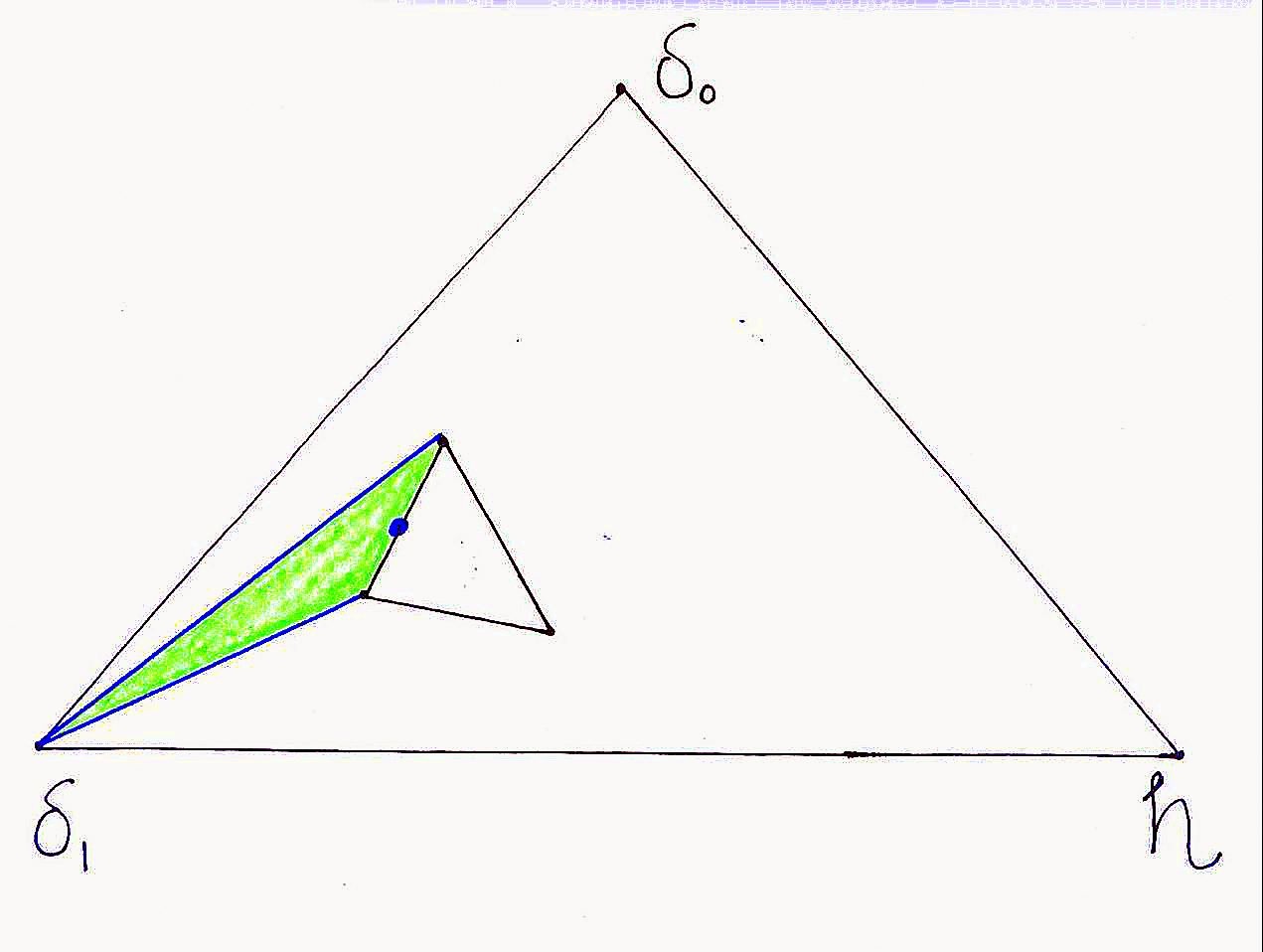}}
{\includegraphics[width=3cm,height=2.5cm]{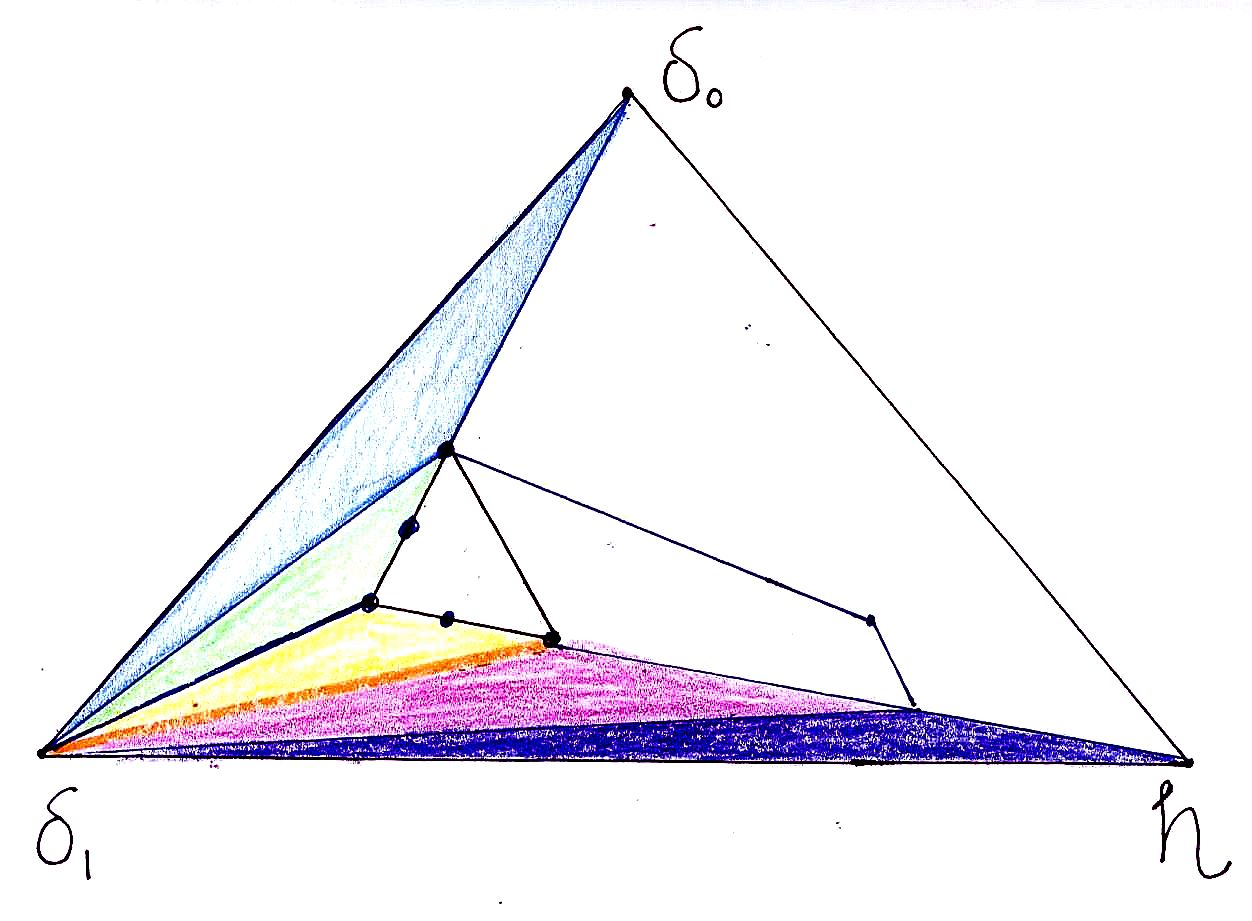}}
  \end{figure}

{\em{Layout of the paper:  }}  Background on divisors and curves on $\ovop{M}_g$ is given in Section \ref{background}.  A combinatorial proof that  $\mathbb{F}$ is a two dimensional extremal face of the nef cone of $\ovop{M}_g$ is given in Section \ref{extremal}.  The face $\mathbb{F}$ is shown to be equal to the pullback of the nef cone of $\op{A}_g^{per}$ along the newly discovered morphism $f_{per}$  in Section \ref{pullback}. Finally, 
the main result, 
Theorem \ref{general}, that any possible extension of the Torelli map is given by a divisor on $\mathbb{F}$  is proved in Section \ref{face}.

 
  

\subsection*{Acknowledgements}

The idea for the project grew out of discussions in the $2010$ MRC program at Snowbird Utah.  The working group Birational Geometry of $\ovop{M}_{g}$, organized by myself, led by Maksym Fedorchuk and Noah Giansiracusa, considered the birational geometry of $\ovop{M}_3$ in great detail and Dave Jensen pointed out that the chamber adjacent to Satake chamber for $g=3$ must be correspond to the morphism from $\ovop{M}_3$ to the Voronoi compactification $\op{A}^{vor}_3$.  I lectured about these results in the first half of the Spring 2011 UGA VIGRE Algebraic Geometry student seminar and I am grateful to the participants Maxim Arap, Patrick McFaddin, Xiaoyan Hu, Lauren Huckuba, Ryan Livingstone, Jaeho Shin, and Joseph Tenini.   I would like to thank Valery Alexeev,  Jesse Kass, and Robert Varley for very helpful discussions about
abelian varieties and their moduli, and  Daniel Krashen for helpful discussions about everything else.  I thank Maxim Arap for his many helpful comments.  The pictures of the effective cone of $\ovop{M}_3$ came from Bill Rulla's thesis.

\section{Background on divisors and curves on $\ovop{M}_g$}\label{background}
  As is conventional, let $\lambda$ be the first Chern class of the Hodge bundle,  and for $0\le i \le \lfloor\frac{g}{2}\rfloor$, let $\delta_i$ be the class of the boundary component $\Delta_i$.  The divisor $\Delta_0$ is the component of the boundary of $\ovop{M}_g$ whose generic element has a single node which is nonseparating, while for $1\le i \le \lfloor\frac{g}{2}\rfloor$, the divisor $\Delta_i$ is the component of the boundary of $\ovop{M}_g$ whose generic element has a single node separating the curve into a component of genus $i$ and a component of genus $g-i$.   For $g \ge 3$, the divisor classes in the set $\{\lambda\} \cup \{\delta_i: 0\le i \le \lfloor\frac{g}{2}\rfloor \}$ form a basis for $\operatorname{Pic}(\ovop{M}_{g})$.

\begin{definition}\label{FCurves}  We will next  recall the definition of the six types of curves on $\ovop{M}_g$, which in the literature are called $F$-curves. The first curve $C_1$ is a family of elliptic tails.  The remaining five types
are obtained by attaching curves to a fixed $4$-pointed stable curve $B$ of genus zero with one point moving and the other three fixed.  The curves to attach are described as follows:

\bigskip

\renewcommand{\arraystretch}{1.2}
\begin{tabular}{lllll}
$C$ & curves to attach to $B$ to obtain $C$\\
\hline
\hline
$C_2$ & a $4$-ptd curve of genus  $g-3$\\ 
\hline
$C^i_{3}$ &  a $1$-ptd curve of genus $i$ and a  $3$-ptd curve of genus $g-2-i$ \\ 
\hline
$C^i_{4}$ &  $2$-ptd curves of genus $i$ and of genus $g-2-i$ \\ 
\hline
$C^{ij}_{5}$ &  $1$-ptd curves of genus $i$ and $j$, a $2$-ptd curve of genus $g-1-i-j$\\ 
\hline
$C^{ijk}_{6}=F_{i,j,k,\ell}$ &   $1$-ptd curves of genus $i$, $j$, $k$, and $\ell=g-(i+j+k)$  \\ 
\hline
\end{tabular}
\renewcommand{\arraystretch}{1}

\bigskip
\noindent In the table above, to define $C^i_{3}$, take $1 \le i \le g-2$; to define $C^i_{4}$, take  $0 \le i \le g-2$; to define $C^{ij}_{5}$, 
take $i$ and $j \ge 1$, such that  $i+j \le g-1$; and to define $C^{ijk}_{6}=F_{i,j,k,\ell}$, take $i$,$j$,$k$,$\ell \ge 1$.  When attaching a $2$-pointed curve of genus zero, replace the resulting rational bridge with a node.

\end{definition}

The table below shows the intersection number for an arbitrary divisor $D=a\lambda-\sum_{i=0}^{\lfloor \frac{g}{2} \rfloor}b_i \delta_i$ with 
the $F$-curves, and the degree of intersection for the divisors $\lambda$ and $12\lambda - \delta_0$, which are the generators of the face $\mathbb{F}$.  

\bigskip
\renewcommand{\arraystretch}{1.2}
\begin{tabular}{lllllll}\label{FInequalities}

$C$ & $ D\cdot C= (a\lambda-\sum_{i=0}^{\lfloor \frac{g}{2} \rfloor}b_i \delta_i) \cdot C$ & $\lambda \cdot C$ & $(12\lambda - \delta_0) \cdot C$ \\ 
\hline
\hline
$C_1$ & $\frac{a}{12}-b_0 + \frac{b_1}{12}$ & $\frac{1}{12}$ & $0$ \\ 
\hline
$C_2$ & $b_0$ & $0$ & $1$ \\ 
\hline
$C^i_{3}$ & $b_i$ & $0$ & $0$ \\ 
\hline
$C^i_{4}$ & $2b_0-b_{i+1}$ & $0$ & $2$ \\ 
\hline
$C^{ij}_{5}$ & $b_i+b_j-b_{i+j}$ & $0$ & $0$ \\ 
\hline
$C^{ijk}_{6}=F_{i,j,k,\ell}$ & $b_i+b_j+b_k+b_{\ell}-b_{i+j}-b_{i+k}-b_{i+\ell}$ & $0$ & $0$ \\ 
\hline

\end{tabular}


\section{An extremal face of the nef cone of $\ovop{M}_g$}\label{extremal}

\begin{proposition}\label{extremalrays}
The divisors $\lambda$ and $12\lambda-\delta_0$
generate independent extremal rays of $\operatorname{Nef}(\ovop{M}_g)$.  The set of divisors 
$$\mathbb{F}=\{\alpha \lambda + \beta (12\lambda - \delta_0): \alpha, \beta \ge 0 \},$$
is a $2$-dimensional extremal face of $\operatorname{Nef}(\ovop{M}_g)$.
\end{proposition}

To prove Proposition \ref{extremalrays}, we show that a particular set of $F$-curves is independent.
\begin{lemma}\label{indcurves}Let $d$ be the dimension of the $NS(\ovop{M}_g)$. The set  $S=\{C_1, C_2\} \cup \{C^i_{3}: 1 \le i \le d-2\}$ consists of $d$ independent curves.
\end{lemma}

\begin{proof}(of Lemma \ref{indcurves})
To show that $S$ is independent, using Table \ref{FInequalities}, we form the matrix of intersection numbers of these curves with the basis $12\lambda$, $12\lambda-\delta_0$, $\delta_1$, $\ldots$, $\delta_{d-2}$ for $\operatorname{Pic}(\ovop{M}_g)$.  Since this is the identity matrix, it shows that the curves in $S$ form a basis for the $1$ cycles of $\ovop{M}_g$.  
\[
\begin{tabular}{c | c c c c c c }
 & $C_1$ & $C_2$ & $C^1_{3}$ & $C^2_{3}$ & $\cdots$ & $C^{d-2}_{3}$ \\
\hline
$12 \lambda$ & $1$ & $0$ & $0$ & $0$ & $\cdots$ & $0$ \\
$12 \lambda - \delta_0$ & $0$ & $1$ & $0$ & $0$ & $\cdots$ & $0$ \\
$\delta_1$ & $0$ & $0$ & $1$ & $0$ & $\cdots$ & $0$ \\
$\delta_2$ & $0$ & $0$ & $0$ & $1$ & $\cdots$ & $0$ \\
\vdots & \vdots & \vdots & \vdots & \vdots & $\ddots$ \\
$\delta_{d-2}$ & $0$ & $0$ & $0$ & $0$ & $\cdots$ & $1$ \\
\end{tabular}
\]

\end{proof}

\begin{proof}(of Propositon \ref{extremalrays})
By \cite{MumfordEnumerative}, and \cite[Prop3.3]{FaberThesis} the divisors $\lambda$, and  $12\lambda - \delta_0$ are nef.
Let $S=\{C_1, C_2\} \cup \{ C^1_{3}, \ldots, C^{d-2}_{3}\}$. Using Table \ref{FInequalities}, one can see that the intersection numbers with the curves defined in Lemma \ref{indcurves} are as follows:
\begin{enumerate}
\item $\lambda \cdot C_1 \ne 0$,  and $\lambda \cdot C = 0$, for all $C \in S \setminus C_1$; 
\item $(12\lambda -\delta_0)\cdot C_2 \ne 0$,  and $(12\lambda -\delta_0) \cdot C = 0$, for all $C \in S \setminus C_2$; 
\item  For $\epsilon >0$,  
$(\epsilon \lambda + 12\lambda - \delta_0) \cdot C_1 \ne 0$, $(\epsilon \lambda + 12\lambda - \delta_0) \cdot C_2 \ne 0$,  and $(\epsilon \lambda + 12\lambda - \delta_0) \cdot C = 0$, for all $C \in S \setminus \{C_1,C_2\}$; 
\end{enumerate}
Since by Lemma \ref{indcurves}, all of the curves are independent, it follows that $\lambda$ and $12\lambda-\delta_0$ generate extremal rays of the nef cone of $\ovop{M}_g$ while positive
 linear combinations of them lie on the interior of a 2-dimensional extremal face.
\end{proof}

We shall see in Section \ref{pullback} that all elements on the face of divisors ${\mathbb{F}}$ are pullbacks of nef divisors along the morphism $f_{per}: \ovop{M}_g \longrightarrow \op{A}_g^{per}$, and in particular, the divisors on the interior of the face are semi-ample.
In this section we discuss the fact that we do not know of another way to show that any element of that face is semi-ample.  In particular
we make the following observation.
 
\begin{remark}\label{bpf}
One cannot trivially use the basepoint-free Theorem to prove any divisor  $D_{\alpha \beta}=\alpha \lambda + \beta (12 \lambda - \delta_0)$, for $\alpha$, and $\beta \ge 0$ is semi-ample. Indeed, recall that the basepoint-free Theorem \cite{KollarMori}*{Thm. 3.3}  says that if $(X,\Delta)$ is a proper klt pair with $\Delta$ effective, and if $D$ is a nef Cartier divisor such that
$mD-K_X-\Delta$ is nef and big for some $m >0$, then $|kD|$ has no basepoints for all $k >>0$.  
By  \cite{BCHM}*{Lemma 10.1}, the pair $(\ovop{M}_{g}, \Delta)$, with
 $\Delta=\sum_{i=0}^{\lfloor \frac{g}{2} \rfloor}\delta_i$ on $\ovop{M}_{g}$ is klt. 
  The canonical divisor of $\ovop{M}_g$ can be expressed as
 $$K_{\ovop{M}_{g}}=13\lambda - 2 \sum_{i=0}^{\lfloor \frac{g}{2} \rfloor}\delta_i.$$
\noindent
It is easy to see, no matter what choice of nonnegative $\alpha$, and $\beta$,  that one can't find $m$ so that 
$mD_{\alpha \beta} - (K_{\ovop{M}_g}+\Delta)$ is nef.  Indeed, 
$$ mD_{\alpha \beta}- (K_{\ovop{M}_g}+\Delta) 
 =(m (\alpha + 12 \beta)-13) \lambda - (m \beta -1)\delta_0 - \sum_{i=1}^{\lfloor \frac{g}{2} \rfloor}(-1)\delta_{i}.$$
\noindent
If this divisor were nef, then it would be an $F$-divisor, which is a divisor that nonnegatively intersects all $F$-curves on $\ovop{M}_g$ .  However, it intersects the curves $C^i_{3}$ in degree $-1$.

\end{remark}

\section{The pullback of the nef cone of $\op{A}_g^{per}$}\label{pullback}
In this section we show that the two dimensional face $\mathbb{F}$ spanned by the extremal divisors $\lambda$ and $12\lambda-\delta_0$ is equal to the pullback of the nef cone of $\op{A}_g^{per}$ along the morphism $f_{per}: \ovop{M}_g \longrightarrow \op{A}_g^{per}$.  This shows that in particular, the divisors on the interior of the face $\mathbb{F}$ are semi-ample.    As $\lambda$ is the pullback of the ample divisor $M$ of weight $1$ modular forms on $\op{A}_g^{Sat}$ along
the morphism  $f_{Sat}: \ovop{M}_g \longrightarrow \op{A}_g^{Sat}$, one has that $\lambda$ is semi-ample.  Ideally one would like to know that $12\lambda- \delta_0$ is also semi-ample.  Using Shepherd-Barron's results about semi-amplenss of the extremal divisor $12M-D_g^{per}$ for low $g$, we show that for $g \le 11$, as long as $\ovop{M}_g$ is defined over $\mathbb{C}$, then this is true.

 We will use the following to prove Proposition \ref{nefpullback} as well as to prove Theorem \ref{general}.
 
 \begin{lemma}\label{keylemma}
For $j:\op{M}_g^{\star}\hookrightarrow \ovop{M}_g$ the embedding of the moduli space $\op{M}_g^{\star}$ of stable curves of genus $g$ of compact type,
one has that $$Ker(j^*: \op{Pic}(\ovop{M}_g)_{\mathbb{Q}}\longrightarrow\op{Pic}(\op{M}_g^{\star})_{\mathbb{Q}}) \cong \mathbb{Q}\delta_0.$$
 \end{lemma}

 \begin{proof}
 Let $X$ be any compactification of $\op{A}_g$ for which there is a morphism $\eta: X \longrightarrow \op{A}_g^{Sat}$,
and such that $\eta^{-1}(\op{A}_g \cup A_{g-1})=\op{A}_g^{part}=\op{A}_g \cup {D}_g$ is Mumford's partial compactification of $\op{A}_g$. 
Suppose that $f: \ovop{M}_g \longrightarrow X$ any morphism that extends the Torelli map.    Note that $f_{Sat}=\eta \circ f$, and $\op{M}_g^{\star}=f_{sat}^{-1}(\op{A}_g)$.  We have the following commutative diagram.

\[\xymatrix @C=2cm @R=1cm {
& & X \ar[rd]^{\eta} \\
	\ovop{M}_g \ar[rru]^{f}  & 
	\op{M}_g^{\star} \ar@{_{(}->}[l]_{ \ \ \  j} \ar[r]^{\phi} \ar[ru]  & 
	\op{A}_g^{part} \ar[r] \ar@{^{(}->}[u]_{i}  & 
	\op{A}_g^{Sat}=\op{A}_g \cup {A}_{g-1} \cup \cdots \cup {A}_{0}.
}\]

\noindent
 The inclusion 
$$\ovop{M}_g \setminus \Delta_0 \stackrel{j}{\hookrightarrow} \op{M}_g^{\star},$$ 
induces an isomorphism 
$$A^1(\op{M}_g^{\star}) \cong A^1(\ovop{M}_g \setminus \Delta_0),$$
giving the following right exact sequence
$$\mathbb{Z} \delta_0 \longrightarrow A^1(\ovop{M}_g) \longrightarrow A^1(\op{M}_g^{\star}) \cong A^1(\ovop{M}_g \setminus \Delta_0) \longrightarrow 0.$$
Tensoring with $\mathbb{Q}$ is exact and so we get the right exact sequence
$$\mathbb{Q} \delta_0 \longrightarrow \operatorname{Pic}(\ovop{M}_g)_{\mathbb{Q}} \overset{j^*}{\longrightarrow}  \operatorname{Pic}(\op{M}_g^{\star})_{\mathbb{Q}}  \longrightarrow 0.$$
In particular, elements in the kernel of $j^*$ are equivalent to rational multiples of $\delta_0$.  

 \end{proof}

\begin{proposition}\label{nefpullback}Let $f_{per}: \ovop{M}_g \longrightarrow \op{A}_g^{per}$ be the Toroidal extension of the Torelli map.  Then
$$f_{per}^*(\operatorname{Nef}(\op{A}_g^{per}))={\mathbb{F}}$$
 
 \end{proposition}

\begin{proof}
There is a morphism $\eta_{per}: \op{A}_g^{per} \longrightarrow \op{A}_g^{Sat}$, and that 
$\op{A}_g^{per}$ contains $\eta_{per}^{-1}(\op{A}_g \cup \op{A}_{g-1}) = \op{A}_g^{part}= \op{A}_g \cup D_g$.   By \cite[Theorem $0.1$]{SBFirst}, $$\operatorname{Nef}(\op{A}_g^{per})=\{aM - b D_g^{per}: a \ge 12 b \ge 0\},$$
where $M=\eta_{per}^*(M)$,  $M$ is the ample generator of $\operatorname{Pic}(\op{A}_g^{Sat})_{\mathbb{Q}}$, and $D_g^{per}$ is the closure of $D_g$ in $\op{A}_g^{per}$.
In particular, $12M - D^{per}_g$ and $M$ generate the nef cone of $\op{A}_g^{per}$.
The result will therefore follow after we show that $f_{per}^*(M) = \lambda$ and that $f_{per}^*(12 M - D^{per}_g) = 12 \lambda - \delta_0$.

Consider the following commutative diagram.

\[\xymatrix @C=2cm @R=1cm {
& & \op{A}_g^{per} \ar[rd]^{\eta_{per}} \\
	\ovop{M}_g \ar[rru]^{f_{per}}  & 
	\op{M}_g^* \ar@{_{(}->}[l]_{ \ \ \  j} \ar[r]^{\phi} \ar[ru]  & 
	\op{A}_g^{part} = {A}_{g} \cup D_g \ar[r] \ar@{^{(}->}[u]_{i}  & 
	\op{A}_g^{Sat}.
}\]

Recall that $f_{Sat}^{-1}(\op{A}_g) =\op{M}_g^{\star}$ is the moduli space of genus $g$ stable curves of compact type  \cite{MumfordGIT}.  
 We denote by $M$ the pull back of $M$ to the two varieties $\op{M}_g^{\star}$, and $\op{A}_g^{part}$.  By Lemma \ref{keylemma}, any element of the kernel of $j^*$ is equivalent to a rational multiple of $\delta_0$.
In particular, if 
$\lambda$, $\delta_0$, $B_3$, $\ldots$, $B_d$ is any basis for $\operatorname{Pic}(\ovop{M}_g)_{\mathbb{Q}}$, then $j^*\lambda$, $j^*B_3$, $\ldots$, $j^*B_d$ is a basis for $\operatorname{Pic}(\op{M}_{g}^{\star})_{\mathbb{Q}}$.

Now we are in the position to prove our assertion.
Let $\mathcal{D}$ be any nef divisor on $\op{A}_g^{per}$, and let  $\{\lambda, \delta_0, B_3, \ldots, B_d\}$ is any basis for $\operatorname{Pic}(\ovop{M}_g)_{\mathbb{Q}}$.
 Then $f_{\tau}^*(\mathcal{D})=a\lambda - b \delta_0 - \sum_{i=3}^d b_i B_i$, and $j^*(f^*\mathcal{D})=aj^*\lambda -  \sum_{i=3}^d b_i j^*B_i$.
We will first show that the $b_i=0$ for $i \ge 3$.  To do this we'll see that $\sum_{i=3}^d b_i j^*B_i \in ker(j^*)$.

Since $\mathcal{D}$ is nef on $\op{A}_g^{per}$,  we may write $\mathcal{D}=\alpha L - \beta D^{per}_g$ such that $\alpha \ge 12 \beta \ge 0$.  
We also have that $i^*(\mathcal{D})=\alpha L - \beta D_g$.  On the other hand, 
  $\phi^*{D}_g=0$, since $\op{M}_g^{\star}=f_{Sat}^{-1}(\op{A}_g)$.
By commutativity of the diagram, it follows therefore that  $j^*f^*(\mathcal{D}) = \phi^*(i^*(\mathcal{D}))= \alpha L.$

Comparing the coefficients of the basis elements:
$$aj^*\lambda -  \sum_{i=3}^d b_i j^*B_i=aL-  \sum_{i=3}^d b_i j^*B_i=\alpha L.$$
And so we see that  $a=\alpha$ and $b_i=0$, for $3 \le i \le d$.  In particular, we have shown that $f^*(L)= \lambda$ and that $f^*(D^{per}_g)=c\delta_0$, for some $c \in \mathbb{Q}$.  We will show that $c$ has to be $1$.

In \cite{SBFirst}, Shepherd-Barron has shown that the Mori Cone of curves of $\op{A}_g^{per}$ is generated by curves 
 $C(1)$ and $C(2)$,
where $C(1)$ is the closure of the set of points $B \times E$ where $B$ is a fixed principally polarized
abelian $(g-1)$-fold and $E$ is a variable elliptic curve, and $C(2)$ is any exceptional
curve of the contraction $\eta_{per}: \op{A}_g^{per} \longrightarrow \op{A}_g^{Sat}$.  
The image of the $\op{F}$-curve $C_1$ on $\ovop{M}_g$, defined in Definiton \ref{FCurves}, under $f_{Sat}$ is the same as the image of $C(1)$ under the map $\eta_{per}$
which gives that $C(1)$ is the image under $f_{per}$ of the F-Curve $C_1$.

Shepherd-Barron showed that $C(1) \cdot (12 L - D_g^{per}) =0$ and we have shown that $C_1  \cdot \lambda \ne 0$.
If $c=0$, then by the projection formula
\begin{multline*}
0=C(1) \cdot (12 L - c D^{per}_g) =(f_{per})_*(C_1) \cdot (12 L - c D_g)\\
=(f_{per})_*(C_1 \cdot f_{per}^*(12 L-cD_g)) =(f_{per})_*(C_1 \cdot f_{per}^*(12 L))=  (f_{per})_*(C_1 \cdot 12 \lambda) \ne 0,
\end{multline*}
\noindent
which is clearly a contradiction.

Therefore, assuming $c\ne 0$, we consider the divisor 
$f_{per}^*(12 L - \frac{1}{c} D_g^{per})= 12 \lambda - \delta_{0}$.
$$0 = C_1 \cdot (12 \lambda - \delta_0 )= C_1 \cdot f_{per}^*(12 L - \frac{1}{c}D^{per}_g) =(f_{per})_*(C_1) \cdot (12 L - \frac{1}{c}D_g^{per}).$$
But  $(f_{per})_*(C_1)=C(1)$ and the kernel of $C(1)$ is a line in the $2$-dimensional vector space $NS(\op{A}_g^{per})_{\mathbb{Q}}$ that contains the point $12 L - D_g^{per}$.  So 
in other words, $12 L - \frac{1}{c} D_g^{per}$ must be a rational multiple of $12 L - D_g^{per}$.  This can only happen if $c=1$.
We have shown that $f_{per}^*L = \lambda$ and that $f_{per}^*(12L - D_g^{vor})= 12\lambda - \delta_0$, and so the result is proved.

\end{proof}

\begin{corollary}\label{semiample}All divisors interior to $\mathbb{F}$ are semi-ample, and when $\ovop{M}_g$ is defined over $\mathbb{C}$,
and  $g \le 11$, the divisor $12\lambda - \delta_0$ on $\ovop{M}_g$ is semi-ample.
\end{corollary}

\begin{proof}By Theorem \ref{nefpullback}, the elements interior to $\mathbb{F}$ are semi-ample. In \cite{SBFirst}, Shepherd-Barron shows that for $g \le 11$, as long as  $\ovop{M}_g$ is defined over $\mathbb{C}$,
then $12M-D_g^{vor}$, which we show pulls back by $f_{per}$ to $12\lambda-\delta_0$,  is semi-ample. \end{proof}

 In \cite{Rulla}, it is shown that every nef divisor on $\ovop{M}_3$ is semi-ample, and so in particular, $12\lambda-\delta_0$ was known to be semi-ample in that case.   
 
\section{The divisors that give the extensions of the Torelli map}\label{face}
In this section, we prove that the pullback of an ample divisor under any extension of the Torelli map lies on $\mathbb{F}$.

\begin{theorem}Let $X$ be any compactification of $\op{A}_g$ that contains Mumford's partial compactification $\op{A}_g^{part}$ and maps to Satake's compactification $\op{A}_g^{Sat}$.  
Then if $f: \ovop{M}_g \longrightarrow X$ is any extension of the Torelli map and $A$ any ample divisor on $X$, there exists a constant $c >0$ and  an $\epsilon > 0$  for which
$f^*(cA)=(12+\epsilon)\lambda - \delta_0$.
\end{theorem}

\begin{proof}  
Let $X$ be any compactification of $\op{A}_g$ for which there is a morphism $\eta: X \longrightarrow \op{A}_g^{Sat}$,
and such that $\eta^{-1}(\op{A}_g \cup \op{A}_{g-1})=\op{A}_g^{part}=\op{A}_g \cup {D}_g$ is Mumford's partial compactification of $\op{A}_g$. 
Suppose that $f: \ovop{M}_g \longrightarrow X$ any morphism that extends the Torelli map.  In particular, 
 $f_{Sat}=\eta \circ f: \ovop{M}_g \longrightarrow \op{A}_g^{Sat}$ denotes the extension of the Torelli map from $\ovop{M}_g$ to the Satake compactification $\op{A}_g^{Sat}$.

 We have the following commutative diagram.

\[\xymatrix @C=2cm @R=1cm {
& & X \ar[rd]^{\eta} \\
	\ovop{M}_g \ar[rru]^{f}  & 
	\op{M}_g^* \ar@{_{(}->}[l]_{ \ \ \  j} \ar[r]^{\phi} \ar[ru]  & 
	\op{A}_g^{part} \ar[r] \ar@{^{(}->}[u]_{i}  & 
	\op{A}_g^{Sat}=\op{A}_g \cup {A}_{g-1} \cup \cdots \cup {A}_{0}.
}\]

\noindent
Note that for $M$, the $\mathbb{Q}$-line bundle of modular forms of weight $1$ on $\op{A}_g^{Sat}$, which is the ample generator of $\operatorname{Pic}(\op{A}_g^{Sat})$, then $f_{Sat}^*(M)= \lambda$ and that $f_{Sat}^{-1}(\op{A}_g) =\op{M}_g^{\star}$ is the moduli space of genus $g$ stable curves of compact type  \cite{MumfordGIT}.  We denote by $M$ the pull back of $M$ to the three varieties $\op{M}_g^{\star}$, 
$\op{A}_g^1$ and $X$.

The assumption that there is an embedding of $\op{A}_g^{part}$ into $X$ tells us that the Picard number of $X$ is at least two, for 
 it is  known that $\operatorname{Pic}(\op{A}_g^{part}) \otimes \mathbb{Q}=\mathbb{Q} D_g \oplus \mathbb{Q} M$ (for $g \in \{2,3\}$,  see \cite{vdg}, and for $g \ge 4$, see \cite[p. 355]{MumfordKodaira}).  In particular, this means that if $A$ is any ample divisor on $X$, then $f^*A$ cannot generate an extremal ray of the nef cone of $\ovop{M}_g$, else the image of $\ovop{M}_g$ under $f$ would have Picard rank $1$.

The goal of this proof is to show that, if $A$ is ample on $X$, then there exists an $\epsilon >0$ for which $f^*A=\epsilon \lambda + 12\lambda - \delta_0$, which is on the interior of $\mathbb{F}$.
By Lemma \ref{keylemma}, elements of the kernel of $j^*$ are equivalent to a rational multiple of $\delta_0$.  Let  $\{\lambda, \delta_0, B_3, \ldots, B_d\}$ is any basis for $\operatorname{Pic}(\ovop{M}_g)_{\mathbb{Q}}$.
 Then $f^*(A)=a\lambda - b \delta_0 - \sum_{i=3}^d b_i B_i$, and $j^*(f^*A)=aj^*\lambda -  \sum_{i=3}^d b_i j^*B_i$.
We will first show that the $b_i=0$ for $i \ge 3$.  To do this we'll see that $\sum_{i=3}^d b_i j^*B_i \in ker(j^*)$.

Since $A$ is ample on $X$, and $i: \op{A}_g^{part} \hookrightarrow X$ is an embedding, $i^*A$ is ample on $\op{A}_g^{part}$.  
In \cite[Prop. I.7]{HulekSankaranNef}, it is proved that the nef cone of $\op{A}_g^{part}$
is given by $$\operatorname{Nef}(\op{A}_g^{part})=\{\alpha {L} -\beta {D}_g: \alpha \ge 12\beta  \ge 0\}.^{\footnote{A nef divisor on a quasi-projective variety $X$ is one that nonnegatively intersects all complete curves on X.}}$$  
So we may write $i^*A=\alpha M - \beta D_g$ such that $\alpha \ge 12 \beta \ge 0$.  But since $M$ is nef but not ample, we also know that $\beta >0$,
so $\alpha > 0$.  Moreover,   $\phi^*{D}_g=0$, since $\op{M}_g^{\star}=f_{Sat}^{-1}(\op{A}_g)$.
It follows therefore that  $j^*f^*(A) = \phi^*(i^*(A))= \alpha L.$  Comparing the coefficients of the basis elements, 
$$aj^*\lambda -  \sum_{i=3}^d b_i j^*B_i=aL-  \sum_{i=3}^d b_i j^*B_i=\alpha L,$$
we see that  $a=\alpha$ and $b_i=0$, for $3 \le i \le d$.

We therefore have that $f^*({A})= a \lambda - b \delta_0$, for $b \in \mathbb{Q}_{\ge 0}$.  Moreover, since $f^*A$ is a nef divisor on $\ovop{M}_g$, it must nonnegatively intersect all $F$-curves and so $a \ge 12b \ge 0$.   If $b=0$, then $f=a \lambda$, which by Proposition \ref{extremal} generates an extremal ray of the nef cone, and so the image of the map $f: \ovop{M}_g \longrightarrow X$ has picard number $1$, contradicting the assumption that
$\op{A}_g^{part}$ is embedded in $X$.

 If $a=12 b$, then  $f^*(\frac{1}{b} A ) = 12 \lambda - \delta_0$, which again by Proposition \ref{extremal} generates an extremal ray of the nef cone, and so 
the image of the morphism $f$ will be a variety of Picard rank $1$, giving a contradiction.  

We conclude that since $a > 12 b>0$, we have that $f^*(\frac{1}{b}A)=\frac{a}{b}\lambda - \delta_0 = (12 + \epsilon) \lambda - \delta_0$, for some $\epsilon >0$.

\end{proof}

\begin{corollary}\label{main}Suppose that $f_{\tau}: \ovop{M}_g \longrightarrow {A}^{\tau}_g$
is any Toroidal extension of the Torelli map.  Then the morphism $f_{\tau}$ is given by a sub-linear series of divisor interior to $\mathbb{F}$.
\end{corollary}

\begin{proof}
Suppose $f_{\tau}: \ovop{M}_g \longrightarrow \op{A}^{\tau}_g$
is any Toroidal extension of the Torelli map.  There is a morphism $\eta_{\tau}: \op{A}^{\tau}_g \longrightarrow \op{A}_g^{Sat}$ containing
Mumford's partial compactification $\op{A}_g^{part}=\eta_{\tau}^{-1}(\op{A}_g \cup {A}_{g-1})$.  So by Theorem \ref{general}, the result holds.
\end{proof}

\begin{remark}The closure, $\ovop{H}_g$ in $\ovop{M}_g$ of the moduli space of hyperelliptic curves $\op{H}_g \subset \op{M}_g$ is equal to the image of the morphism $h: \ovop{M}_{0,2(g+1)}/S_{2(g+1)} \longrightarrow \ovop{M}_{g}$
defined by taking $(C,\ov{p})\in \ovop{M}_{0,2(g+1)}$
to the stable $n$-pointed curve of genus $g$ obtained by taking a double cover of $C$ branched at the (unordered) marked points.
For $\epsilon=\frac{1}{3\cdot 2^{g-5}}$, and $c=\frac{1}{4}$, one has that $\kappa_1 = h^*(c(\epsilon \lambda + 12\lambda - \delta_0)).$
In particular,  the morphism on $\ovop{M}_{0,2(g+1)}/S_{2(g+1)}$ given by $\kappa_1$ factors through $h$.
  \end{remark}

\begin{remark} In Corollary \ref{semiample} it was shown that a multiple of $12\lambda-\delta_0$ is base point free for $g \le 11$.  It is natural
to look for 
 the variety $Y$ (which must necessarily have Picard number $1$) and the corresponding morphism $f_{m|12\lambda-\delta_0|}: \ovop{M}_g \longrightarrow Y$.
 \end{remark}

\end{document}